\documentclass{amsart}

\newtheorem{theorem}{Theorem}[section]
\newtheorem{lemma}[theorem]{Lemma}

 \theoremstyle{definition}

\theoremstyle{remark}
\newtheorem{remark}[theorem]{Remark}

\newcommand{\ul}{\underline}

\numberwithin{equation}{section}

\begin{document}

\title[Degenerate Hessian equations]
{The Dirichlet problem for a class of Hessian type equations}

\author{Heming Jiao}
\address{Department of Mathematics, Harbin Institute of Technology,
         Harbin, 150001, China}
\email{jiao@hit.edu.cn}
\author{Tingting Wang}
\address{Department of Mathematics, Harbin Institute of Technology,
         Harbin, 150001, China\\
         \emph{Present Address:}
         Jiuquan Satellite Launch Center, Jiuquan, 735000, China}
\email{ttwanghit@gmail.com}

\begin{abstract}
We are concerned with the Dirichlet problem for a class of
Hessian type equations. Applying some new methods we are able to establish the
$C^2$ estimates for an approximating problem under essentially optimal
structure conditions. Based on these estimates, the existence of
classical solutions is proved.


\end{abstract}

\maketitle

\section{Introduction}

Let $\Omega$ be a bounded domain in $\mathbb{R}^{n}$ ($n\geq 2$) with smooth boundary $\partial \Omega$.
In this paper, we are concerned with the regularity for solutions of the Dirichlet problem
\begin{equation}
\label{1-1}
\left\{ \begin{aligned}
    f \big(\lambda[D^{2} u + \gamma \triangle u I]\big) & = \psi  \;\;\mbox{ in }~ \Omega, \\
                 u &= \varphi  \;\;\mbox{ on }~ \partial \Omega,
\end{aligned} \right.
\end{equation}
where
$\gamma \geq 0$ is a constant, $I$ is the unit matrix and $\lambda[D^{2} u + \gamma \triangle u I]=(\lambda_{1},...,\lambda_{n})$
denote the eigenvalues of the matrix $\{D^{2} u + \gamma \triangle u I\}$.

Following \cite{CNS3}, $f \in C^2 (\Gamma) \cap C (\bar{\Gamma})$ is assumed to be
defined in an open convex symmetric cone $\Gamma$,
with vertex at the origin and
\[
\Gamma \supseteq \Gamma_{n} \equiv \big\{\lambda \in \mathbb{R}^{n}:
\mbox { each component} ~\lambda_{i} > 0 \big\},
\]
and to satisfy the following structure conditions:
\begin{equation}
\label{1-2}
f_{i} \equiv \frac{\partial f}{\partial \lambda_{i}} > 0 \mbox{ in } \Gamma, 1 \leq i \leq n,
\end{equation}
\begin{equation}
\label{1-3}
f \mbox{ is concave in } \Gamma,
\end{equation}
and
\begin{equation}
\label{1-4}
f > 0 \mbox{ in } \Gamma, ~~ f = 0 \mbox{ on } \partial \Gamma.
\end{equation}

A function $u\in C^{2}(\Omega)$ is called {\it admissible} if
$\lambda[D^{2}u+\gamma\Delta u I]\in \bar{\Gamma}$.
According to \cite{CNS3}, condition (\ref{1-2}) ensures that equation
(\ref{1-1}) is degenerate elliptic for admissible solutions.
While (\ref{1-3}) implies that the function $F$
defined by $F[A]=f(\lambda [A])$ to be concave for $A \in \mathcal{S}^{n\times n}$ with $\lambda[A]\in \Gamma$, where
$\mathcal{S}^{n\times n}$ is the set of $n$ by $n$ symmetric
matrices.

We assume that $\psi \geq 0$ in $\Omega$, so the equation
\eqref{1-1} is degenerate when $\gamma = 0$. In this paper, there are no geometric
restrictions to $\partial \Omega$ being made.
Instead, we assume that there exists a subsolution
$\underline u \in C^2 (\bar{\Omega})$ satisfying
$\lambda (D^2 \underline{u}+ \gamma \triangle \underline{u} I) \in \Gamma$
on $\bar{\Omega}$ and
\begin{equation}
\label{1-5}
\left\{
\begin{aligned}
f (\lambda (D^{2}\underline{u}+\gamma \triangle \underline{u} I))&\geq\psi  \;\;{\mathrm{in}~ \Omega,} \\
\underline{u} &=\varphi \;\;\mathrm{on}~ \partial \Omega.
\end{aligned}
\right.
\end{equation}

\begin{theorem}
\label{jw-th1}
Let $\gamma > 0$, $\psi \in C^\infty (\bar{\Omega})$ and $\varphi \in C^\infty (\partial \Omega)$.
Under \eqref{1-2}-\eqref{1-5}, there exists a unique admissible solution
$u \in C^{\infty} (\bar{\Omega})$ of \eqref{1-1}.
\end{theorem}

We first introduce our procedure to prove Theorem \ref{jw-th1}.
By \eqref{1-4}, there exists a positive constant $\varepsilon_0$ such that
\begin{equation}
\label{sub}
f (\lambda(D^{2}\underline{u}+\gamma \triangle \underline{u} I)) \geq \varepsilon_0 \mbox{ on } \bar{\Omega}
\end{equation}
since $\lambda (D^{2} \ul u + \gamma \triangle \ul u I) \in \Gamma$.
We shall establish the \emph{a priori} $C^2$ estimates independent of
$\varepsilon$ for admissible solutions of the approximating problem
\begin{equation}
\label{appro}
\left\{
\begin{aligned}
f (\lambda (D^{2}u_{\varepsilon}+\gamma\Delta u_{\varepsilon} I))&=\psi+ \varepsilon \eta (\psi)  \;\;{\mathrm{in}~ \Omega,} \\
u_{\varepsilon} &=\varphi \;\;\mathrm{on}~ \partial \Omega,
\end{aligned}
\right.
\end{equation}
where $\eta \in C^\infty [0, \infty)$ satisfies
\[
\eta (t) = \left\{
\begin{aligned}
& 1 \;\;\mathrm{ if } \ \ t \in [0, \frac{\varepsilon_0}{4}],\\
& 0 \;\;\mathrm{ if } \ \ t \in [\frac{\varepsilon_0}{2}, \infty),
\end{aligned}
\right.
\]
$0 \leq \eta \leq 1$, $|\eta'| \leq C \varepsilon_0^{-1}$ and $|\eta''| \leq C \varepsilon_0^{-2}$.
It follows that, by \eqref{sub},
\[
f (\lambda(D^{2}\underline{u}+\gamma \triangle \underline{u} I)) \geq \psi+ \varepsilon \eta (\psi)
\]
provided $\varepsilon \leq \frac{\varepsilon_0}{2}$ and obviously,
$\psi+ \varepsilon \eta (\psi) \geq \min \{\varepsilon, \varepsilon_0/4\} > 0$.

We shall use the techniques of Guan \cite{Guan14} (see \cite{GJ14} and \cite{GSS14} also)
to establish such estimates.
As usual, the main difficulty is from the boundary estimates
of pure normal second order derivative for which we use
the strategy of Ivochkina, Trudinger and Wang \cite{ITW2004}
whose idea is originally from Krylov \cite{Krylov84,Krylov94,Krylov95a,Krylov95b}
where the Bellman equations are studied.
A key step is
the construction of barrier functions in which the existence
of $\ul u$ plays an important role (see Theorem \ref{barrier}).

The presence of $\gamma > 0$ is crucial to
the interior estimates for second derivatives. An interesting
question is to establish the weak interior estimates (see \cite{ITW2004})
when $\gamma = 0$.

For the case that $\psi \geq \psi_0 > 0$,
the existence of smooth solutions to
the Dirichlet problem (\ref{1-1}) with $\gamma=0$
was established by Caffarelli, Nirenberg and Spruck \cite{CNS3}
under additional assumptions on $f$ in a domain $\Omega$
satisfying that there exists a sufficiently
large number $R > 0$ such that, at every point $x \in \partial \Omega$,
\begin{equation}
\label{dombd}
(\kappa_1, \ldots, \kappa_{n-1}, R) \in \Gamma,
\end{equation}
where $\kappa_1, \ldots, \kappa_{n-1}$ are the principal curvatures of $\partial \Omega$
with respect to the interior normal.
Their work was further developed and simplified by Trudinger \cite{Trudinger95}.

Guan considered the Hessian equations of the form
\begin{equation}
\label{1-1'}
f (\lambda [\nabla^2 u + \gamma \triangle u g + s du \otimes du
  - \frac{t}{2} |\nabla u|^2 g + A]) = \psi (x, u, \nabla u)
\end{equation}
on a Riemannian manifold with metric $g$ with $\psi > 0$, which is arising from conformal
geometry (see \cite{Guan2007} and \cite{Guan2008}). In these papers
Guan also assumed that $f$ is homogenous of degree one which implies
that the equation \eqref{1-1'} is strictly elliptic.
It would be interesting to prove Theorem \ref{jw-th1}
for the general form \eqref{1-1'} on manifolds
when $\psi \geq 0$ without any additional conditions on $f$.
The case that $\gamma = 0$ seems more complicated.
In a recent work \cite{Guan14}, Guan proved Theorem \ref{jw-th1}
under \eqref{1-2}-\eqref{1-5}
when $\gamma = 0$ and $\psi \geq \psi_0 > 0$.
Another interesting question would be whether we can get a viscosity
solution in $C^{1,1} (\bar{\Omega})$ for $\gamma = 0$ when $\psi \geq 0$.


It was shown in \cite{CNS3} that using
\eqref{dombd} and the condition that for every $C > 0$ and every compact set $K$ in $\Gamma$
there is a number $R = R (C, K)$ such that
\begin{equation}
\label{f9}
f (R \lambda) \geq C \mbox{  for all } \lambda \in K
\end{equation}
one can construct admissible strict subsolutions of equation \eqref{1-1} with $\gamma = 0$.
Obviously $\Gamma \subset \{\lambda \in \mathbb{R}^n: \sum \lambda_i > 0\}$
and we have $\triangle u \geq 0$ for any admissible function $u$. So we can
construct an admissible strict subsolution of \eqref{1-1}
when $\gamma \geq 0$ satisfying \eqref{1-5}
under \eqref{dombd} and \eqref{f9} by the same way.

Typical examples are given by $f = \sigma^{1/k}_k$ and $f = (\sigma_k / \sigma_l)^{1/(k - l)}$,
$1 \leq l < k \leq n$, defined in the G{\aa}rding cone
\[
\Gamma_{k} = \{\lambda \in \mathbb{R}^{n}: \sigma_{j} (\lambda) > 0, j = 1, \ldots, k\},
\]
where $\sigma_{k}$ are the elementary symmetric functions
\[
\sigma_{k} (\lambda) = \sum_ {i_{1} < \ldots < i_{k}}
\lambda_{i_{1}} \ldots \lambda_{i_{k}},\ \ k = 1, \ldots, n.
\]

The case when $f = \sigma^{1/n}_n$ (the Monge-Amp\`{e}re equation) and
$\gamma = 0$ was studied by Guan, Trudinger and Wang \cite{GTW99}
and they obtained the $C^{1, 1}$ regularity as
$\psi^{1/(n-1)} \in C^{1, 1} (\bar{\Omega})$. It would be
an interesting problem to show whether the result can be improved for the
$f = \sigma^{1/k}_k$ (see \cite{ITW2004}).

The rest of this paper is organized as follows.
In Section 2, we prove Theorem \ref{jw-th1} provided the
$C^2$ estimates for \eqref{appro} is established.
$C^1$ estimate is treated in Section 3.
The interior second order estimate is proved in Section 4.
In section 5, the estimates for second derivatives are established.

\section{Beginning of proof}

In this Section we explain how to prove Theorem \ref{jw-th1} when the second order
estimates for \eqref{appro} are established.
Let $u_\varepsilon \in C^4 (\bar{\Omega})$ be the admissible solution of \eqref{appro}.
For simplicity we shall use the notations
$U^\varepsilon = D^{2} u_\varepsilon + \gamma \triangle u_\varepsilon I$
and $\underline U = D^{2} \underline u + \gamma \triangle \underline u I$.
Following the literature, unless otherwise noted, we denote throughout this paper
\[ F^{ij} [U^\varepsilon] = \frac{\partial F}{\partial U^\varepsilon_{ij}} [U^\varepsilon], \;\;
 F^{ij, kl} [U^\varepsilon]
   = \frac{\partial^2 F}{\partial U^\varepsilon_{ij} \partial U^\varepsilon_{kl}} [U^\varepsilon]. \]
The matrix $\{F^{ij}\}$ has eigenvalues $f_1, \ldots, f_n$ and
is positive definite by assumption \eqref{1-2}, while \eqref{1-3}
implies that $F$ is a concave function of $U^\varepsilon_{ij}$ (see \cite{CNS3}).
Moreover, when $U^\varepsilon$ is diagonal so is $\{F^{ij}\}$, and the
following identities hold
\[
F^{ij} U^\varepsilon_{ij} = \sum f_i \lambda_i, \;\;
F^{ij} U^\varepsilon_{ik} U^\varepsilon_{kj} = \sum f_i \lambda_i^2, \;\;
\lambda [U^\varepsilon] = (\lambda_1, \ldots, \lambda_n).
\]

Suppose $\gamma > 0$ and
we have proved that there exists a constant independent of $\varepsilon$ such that
\begin{equation}
\label{1-7}
|u_\varepsilon|_{C^2 (\bar{\Omega})} \leq C.
\end{equation}

Therefore, by the concavity of $F$,
\[
F^{ij} [U_\varepsilon] (A \delta_{ij} - U^\varepsilon_{ij}) \geq F [A I] - F [U_\varepsilon]
   \geq c_0 > 0
\]
by fixing $A$ sufficiently large. On the other hand, $-F^{ij} U^\varepsilon_{ij} \leq C \sum F^{ii}$
by \eqref{1-7}. Then we get
\[
\sum F^{ii} \geq \frac{c_0}{A + C} > 0.
\]
Note that
\[
\{\frac{\partial F}{\partial u^\varepsilon_{ij}} [U_\varepsilon]\} = \{F^{ij} [U_\varepsilon]\}
  + \gamma \sum F^{ii} I \geq \frac{\gamma c_0}{A + C} I.
\]
Thus, there exists uniform constants $0 < \lambda_0 \leq \Lambda_0 < \infty$ such that
\[
\lambda_0 I \leq \{\frac{\partial F}{\partial u^\varepsilon_{ij}} [U_\varepsilon]\} \leq \Lambda_0 I.
\]
Hence Evans-Krylov theory (see \cite{Evans82} and \cite{Krylov83}) assures a bound $M$
independent of $\varepsilon$ such that
\[
|u_\varepsilon|_{C^{2, \alpha} (\bar{\Omega})} \leq M,
\]
for some constant $\alpha \in (0, 1)$. The higher regularity can be derived by the
Schauder theory (see \cite{GT83} for example).
Using standard method of continuity, we can obtain the existence
of smooth solution to \eqref{appro}.
By sending $\varepsilon$ to zero (taking a subsequence if necessary),
we can prove Theorem \ref{jw-th1}.

In the following sections, we may drop the subscript $\varepsilon$ when
there is no possible confusion.

\bigskip

\section{The gradient estimates}

In this section, we consider the gradient estimates for the admissible
solution to \eqref{appro}.
We first observe that $\lambda[U]\in\Gamma\subset\{\sum\lambda_{i}>0\}$
and therefore,
\begin{equation}\label{2-1}
tr[U]=(1+n\gamma)\Delta u>0.
\end{equation}
Thus we have by the maximum principle that
\begin{equation*}
\underline{u}\leq u\leq h  \;\;~\mathrm{in} ~\bar{\Omega}
\end{equation*}
where $h$ is the harmonic function in $\Omega$ with $h=\varphi$ on $\partial \Omega$. Then we obtain
\begin{equation}\label{2-2}
\sup_{\bar{\Omega}}|u|+\sup_{\partial \Omega}|D u|\leq C,
\end{equation}
for some positive constant $C$ independent of $\varepsilon$.

To establish the global gradient estimates, we assume that $|D u| e^{\phi}$
achieves a maximum at an interior point $x_0 \in \Omega$,
where $\phi$ is a function to be determined.
We may assume
$D^2 u$ and $\{F^{ij}\}$ are diagonal at $x_0$ by
rotating the coordinates if necessary. Then at $x_0$ where the function
$\log |D u| + \phi$ attains its maximum, we have
\begin{equation}
\label{eq4-0-1}
 \frac{u_k u_{ki}}{|D u|^2} + \phi_i = 0
\end{equation}
and
\begin{equation}
\label{eq4-0-3}
\frac{u_k u_{kii} + u_{ki} u_{ki}}{|D u|^2}
  - 2 \frac{(u_k u_{ki})^2}{|D u|^4} + \phi_{ii} \leq 0
\end{equation}
for each $i = 1, \cdots, n$.
Differentiating the equation \eqref{appro}, we get, at $x_0$,
\begin{equation}
\label{dif-1}
F^{ii} u_{kii} + \gamma \triangle u_k \sum F^{ii}
  = \psi_k + \varepsilon \eta' \psi_k.
\end{equation}
It follows that
\begin{equation}
\label{dif-1'}
F^{ii} u_k u_{kii} + \gamma u_k \triangle u_k \sum F^{ii} \geq - C |D u|.
\end{equation}
Note that
\begin{equation}
\label{eq4-0-10}
\begin{aligned}
U^2_{ii} = (u_{ii} + \gamma \triangle u)^2
     \leq \,& 2 u^2_{ii} + 2 \gamma^2 (\triangle u)^2
        \leq 2 u^2_{ii} + 2 n \gamma^2 \sum_j u^2_{jj}\\
          \leq \,& 2 n \max \{\gamma, 1\} (u^2_{ii} + \gamma \sum_j u^2_{jj}).
\end{aligned}
\end{equation}
Therefore, by \eqref{eq4-0-1}, \eqref{eq4-0-3}, \eqref{dif-1'} and \eqref{eq4-0-10}, we have
\begin{equation}
\label{eq4-1}
\begin{aligned}
& c_0 F^{ii} U_{ii}^2
   + |D u|^2 \Big(F^{ii} \phi_{ii} + \gamma \triangle \phi \sum F^{ii}\Big)\\
  \leq \,& C |D u| + 2 |D u|^2 \Big(F^{ii} \phi_i^2 + \gamma |D \phi|^2 \sum F^{ii}\Big),
\end{aligned}
\end{equation}
where $c_0 = (2 n \max \{\gamma, 1\})^{-1}$.

Let $v = \ul u - u + \inf_{\bar{\Omega}} (u - \ul u) + 1$ and $\phi = \frac{\delta v^2}{2}$,
where $\delta$ is a positive constant to be determined. Choosing $\delta$ sufficiently small,
we can  guarantee that
\[
\delta - 2 \delta^2 v^2 > 0 \mbox{ on } \bar{\Omega}.
\]
Let $c_1 = \min_{x \in \bar{\Omega}} \Big(\delta - 2 \delta^2 v^2 (x)\Big) > 0$.
It follows from \eqref{eq4-1}
that
\begin{equation}
\label{eq4-2}
\begin{aligned}
& c_0 F^{ii} U_{ii}^2
   + |D u|^2 \mathcal{L} (\ul u - u)\\
  \leq \,& C |D u| - (\delta - 2 \delta^2 v^2) |D u|^2 \Big(F^{ii} v_i^2 + \gamma |D v|^2 \sum F^{ii}\Big)\\
   \leq \,& C |D u| - c_1 |D u|^2 \Big(F^{ii} v_i^2 + \gamma |D v|^2 \sum F^{ii}\Big).
\end{aligned}
\end{equation}

Write
$\mu (x) = \lambda (D^2 \ul u (x) + \gamma \triangle \ul u (x) I)$ and note
that $\{\mu (x): x \in \bar{\Omega}\}$ is a compact subset of
$\Gamma$. There exists uniform constant $\beta \in (0, \frac{1}{2
\sqrt{n}})$ such that
\begin{equation}
\label{eq2-16}
     \nu_{\mu (x)} -  2 \beta {\bf 1} \in \Gamma_n,
     \;\;  \forall \, x \in \bar{\Omega}
\end{equation}
where ${\bf 1} = (1, \ldots, 1) \in \mathbb{R}^n$ and
$\nu_{\lambda} := Df (\lambda)/|Df (\lambda)|$ is the unit
normal vector to the level hypersurface $\partial \Gamma^{f (\lambda)}$ for
$\lambda \in \Gamma$. We need the following lemma proved by Guan in \cite{Guan14}.
\begin{lemma}
\label{ma-lemma-C10}
For any fixed $x\in\bar{M}$, denote
$\tilde{\mu} = \mu (x)$ and $\tilde{\lambda} = \lambda (U
(x))$. Suppose that $|\nu_{\tilde{\mu}}-
\nu_{\tilde{\lambda}}| \geq \beta$. Then there exists a uniform constant
$\theta > 0$ such that
\begin{equation}
\label{gj-C160'}
  \sum f_i (\tilde{\lambda}) (\tilde{\mu}_i - \tilde{\lambda}_i)
   \geq \theta \Big( 1 + \sum f_i (\tilde{\lambda})\Big).
\end{equation}
\end{lemma}

Now let ${\mu} = \lambda (D^2 \ul u (x_0) + \gamma \triangle \ul u (x_0))$,
$\lambda = \lambda (D^2 u (x_0) + \gamma \triangle u (x_0))$ and $\beta$
as in \eqref{eq2-16}. Suppose first that $|\nu_{\mu}- \nu_{\lambda}|
\geq \beta$. Define the linear operator $\mathcal{L}$ by
\[
\mathcal{L} v := F^{ij} v_{ij} + \gamma \triangle v \sum F^{ii}
\]
for $v \in C^2 (\Omega)$.
By Lemma~\ref{ma-lemma-C10},
\begin{equation}
\label{eq0-0}
\mathcal{L} (\ul u - u)
    \geq \theta \Big(\sum F^{ii} + 1\Big)
\end{equation}
for some $\theta > 0$. Thus, we can obtain a bound $|D u (x_0)| \leq C / \theta$
from \eqref{eq4-2}.

We now consider the case $|\nu_{\mu}- \nu_{\lambda}| < \beta$
which implies $\nu_{\lambda} - \beta {\bf 1} \in \Gamma_n$
and therefore
\begin{equation}
\label{eq2-22}
 F^{ii} \geq \frac{\beta}{\sqrt{n}} \sum F^{kk},
    \;\; \forall \, 1 \leq i \leq n.
\end{equation}
By the concavity of $F$ we know that
\begin{equation}
\label{eq0-1}
\mathcal{L} (\ul u - u) \geq 0.
\end{equation}
By the concavity of $f$ again, when $|\lambda| \geq R$
for $R$ sufficiently large, we derive as in \cite{Guan14}
\begin{equation}
\label{eq4-5}
\begin{aligned}
  |\lambda| \sum F^{ii}
  \geq \,& f (|\lambda| {\bf 1}) - f (\lambda)
        + \sum F^{ii} \lambda_i \\
  \geq \,& f (|\lambda| {\bf 1}) - f (\mu)
       - |\lambda| \sum F^{ii}\\
  \geq \,& 2 b_0 - |\lambda| \sum F^{ii}.
\end{aligned}
\end{equation}
for some uniform positive constant $b_0$.
Therefore, by \eqref{eq2-22} and \eqref{eq4-5}, we find
\begin{equation}
\label{eq4-7}
\begin{aligned}
c_0 F^{ii} U_{ii}^2 + c_1 |D u|^2 F^{ii} v_i^2
  \geq \,& \frac{\beta}{\sqrt{n}} \Big(c_0 |\lambda|^2 \sum F^{ii} + \frac{c_1}{2} |D u|^4 \sum F^{ii}\Big)\\
  \geq \,& \frac{\sqrt{2 c_0 c_1} \beta}{\sqrt{n}}  |D u|^2 |\lambda| \sum F^{ii} \\
  \geq \,& c_2 |D u|^2
\end{aligned}
\end{equation}
provided $|D u|$ is sufficiently large, where $c_2 = \frac{\sqrt{2 c_0 c_1} \beta b_0}{\sqrt{n}}$.
Thus, from \eqref{eq4-2} and \eqref{eq4-7} we can get a bound
$|D u| \leq C / c_2$.

Suppose $|\lambda| \leq R$. By the concavity of $F$, we have (see \cite{GSS14})
\[
2 R \sum F^{ii} \geq F^{ii} U_{ii} + F (2 R I) - F (U)
   \geq - R \sum F^{ii} + b_1,
\]
where $b_1 = F (2 R I) - F (R I) > 0$. It follows that
\begin{equation}
\label{gd-0}
\sum F^{ii} \geq \delta_0 \equiv \frac{b_1}{3 R}
\end{equation}
and
\[
c_1 |D u|^2 F^{ii} v_i^2 \geq \frac{c_1 \beta}{2 \sqrt{n}} |D u|^4 \sum F^{ii}
   \geq \frac{c_1 \beta \delta_0}{2 \sqrt{n}} |D u|^4
\]
provided $|D u|$ is sufficiently large.
We then obtain from \eqref{eq4-2} that $|D u (x_0)| \leq (2 \sqrt{n} C / c_1 \beta \delta_0)^{1/3}$.

Hence we have proved that
\begin{equation}\label{2-5}
|u|_{C^1 (\bar{\Omega})} \leq C
\end{equation}
for some positive constant $C$ independent of $\varepsilon$.

\medskip

\section{Interior and global estimates for second derivatives}

In this section, we prove the interior second order estimate.

\begin{theorem}
\label{jw-th2}
Let $\gamma > 0$ and $u \in C^4 (\Omega)$ be an admissible solution
of \eqref{appro}. Then for any $\Omega' \subset\subset \Omega$,
there exists a constant $C$ depending on $\gamma^{-1}$,
$d' \equiv \mathrm{dist} (\Omega', \partial \Omega)$,
$|u|_{C^1 (\bar{\Omega})}$
and other known data such that
\begin{equation}
\label{3-1}
\sup_{\bar{\Omega}'} |D^2 u| \leq C.
\end{equation}
\end{theorem}

\begin{proof}
Let
\[
W = \max_{x \in \bar{\Omega}, |\xi| = 1} \zeta (x) e^{\phi (x)} D_{\xi\xi} u (x)
\]
where $\zeta$ and $\phi$ are functions to be determined with $\zeta$ satisfying
\begin{equation}
\label{zeta}
0 \leq \zeta \leq 1,\ \ |D \zeta| \leq a_0,\ \ |D^2 \zeta| \leq a_0\ \
   \mbox{ on } \bar{\Omega}.
\end{equation}
Assume that $W$ is achieved at
$x_0 \in \Omega$ and $\xi_0 = e_1 = (1, 0, \ldots, 0)$.
We may also assume that $D^2 u$ is diagonal at $x_0$.
We have, at $x_0$ where the function $\log u_{11} + \log \zeta + \phi$
attains its maximum,
\begin{equation}\label{3-2}
\frac{u_{11i}}{u_{11}} + \frac{\zeta_i}{\zeta} + \phi_i = 0
\mbox{  for each $i = 1, \ldots, n$},
\end{equation}
\begin{equation}\label{3-3}
F^{ii} \Big\{\frac{u_{11ii}}{u_{11}}
   - \Big(\frac{u_{11i}}{u_{11}}\Big)^2
   - \Big(\frac{\zeta_i}{\zeta}\Big)^2  + \frac{\zeta_{ii}}{\zeta} + \phi_{ii}\Big\} \leq 0.
\end{equation}
and
\begin{equation}\label{3-3-0}
\frac{\triangle u_{11}}{u_{11}}
   - \sum_i \Big(\frac{u_{11i}}{u_{11}}\Big)^2
   - \sum_i \Big(\frac{\zeta_i}{\zeta}\Big)^2 + \frac{\triangle \zeta}{\zeta} + \triangle \phi \leq 0.
\end{equation}

Differentiating equation \eqref{appro} twice, by the concavity of $F$, we obtain at $x_{0}$,
\begin{equation}\label{3-4}
F^{ii} u_{ii11} + \gamma (\Delta u)_{11} \sum F^{ii} = \psi_{11} + \varepsilon (\eta' \psi_{11} + \eta'' \psi_1^2) \geq - C.
\end{equation}
Let
\[ \phi = \frac{\delta |D u|^2}{2}, \]
where $\delta > 0$ is a undetermined constant.
By straightforward calculation, we have
\[
\phi_i =\delta u_i u_{ii}
\]
and
\[
\phi_{ii}
     = \delta  u_{ii}^2 + \delta u_{j} u_{jii}.
\]
Note that
\begin{equation}
\label{ib-1}
F^{ii} u_{j} u_{jii} + \gamma u_{j} \Delta u_j \sum F^{ii} = u_j (\psi_j + \varepsilon \eta' \psi_j) \geq - C
\end{equation}
and
\begin{equation}\label{3-8}
\phi_i^2 \leq C \delta^2 u_{ii}^2.
\end{equation}
We have
\begin{equation}\label{3-9}
\begin{aligned}
\mathcal{L} \phi
     \geq \delta F^{ii} u_{ii}^2 + \gamma \delta \sum u_{jj}^2 \sum F^{ii} - C \delta.
\end{aligned}
\end{equation}
Combining \eqref{3-2}-\eqref{3-9}, we get
\begin{equation}
\label{3-10}
\begin{aligned}
0 \geq \,& - \frac{C}{u_{11}} - C \delta + (\delta - C \delta^2) F^{ii} u_{ii}^2\\
  & + \gamma (\delta - C \delta^2) \sum u_{jj}^2 \sum F^{ii}
   - \frac{C}{\zeta^2} \sum F^{ii}.
\end{aligned}
\end{equation}
Choose $\delta$ sufficiently small such that
$\delta - C \delta^2 > 0$.
Let $\lambda = \lambda (D^2 u (x_0) + \gamma \triangle u (x_0))$.
By \eqref{eq4-0-10}, we find that
\[
|\lambda|^2 = \sum U_{ii}^2 \leq 2 n^2 \Big(\max\{\gamma, 1\}\Big)^2 \sum u_{jj}^2.
\]
Thus, it follows from \eqref{3-10} that
\begin{equation}
\label{3-1000}
\begin{aligned}
0 \geq - \frac{C}{u_{11}} - C \delta
   + 2 c_3 |\lambda|^2 \sum F^{ii} - \frac{C}{\zeta^2} \sum F^{ii},
\end{aligned}
\end{equation}
where
\[
c_3 = \frac{1}{4} (\delta - C \delta^2) \gamma n^{-2} \Big(\max\{\gamma, 1\}\Big)^{-2} > 0.
\]
By \eqref{eq4-5} and \eqref{3-1000}, we have
\begin{equation}
\label{3-1001}
0 \geq \Big(b_0 c_3 |\lambda| - \frac{C}{u_{11}} - C \delta\Big)
   + \Big(c_3 |\lambda|^2 - C b^2 - \frac{C}{\zeta^2} \Big) \sum F^{ii}
\end{equation}
provided $|\lambda|$ is sufficiently large.
It follows that $|\lambda| \zeta (x_0) \leq C$.

The function $\zeta$ may now be chosen as a cutoff function satisfying
$\zeta \equiv 1$ on $\Omega' \subset\subset \Omega$ and $|D \zeta| \leq C/d'$,
$|D^2 \zeta| \leq C/d'^2$. Then
\[
|D^2 u| \zeta \leq C \mbox{ on } \bar{\Omega}
\]
and \eqref{3-1} holds.
\end{proof}
\begin{remark}
We remark that in the proof of Theorem \ref{jw-th2} we do not need
the existence of $\ul u$.
\end{remark}

In the proof of Theorem \ref{jw-th2}, setting $\zeta \equiv 1$, we can prove
the following maximal principle.
\begin{theorem}
\label{jiao-th1}
Let $u \in C^4 (\bar{\Omega})$ be an admissible solution
of \eqref{appro}. Then
\begin{equation}
\label{jiao-3-1}
\sup_{\bar{\Omega}} |D^2 u| \leq C (1 + \sup_{\partial \Omega} |D^2 u|),
\end{equation}
where $C$ depends on $\gamma^{-1}$,
$|u|_{C^1 (\bar{\Omega})}$
and other known data.
\end{theorem}

We are interested in the case that $\gamma = 0$. Now we prove \eqref{jiao-3-1}
under the existence of $\ul u$ satisfying \eqref{1-5}
and we will see that the constant $C$ would not depend on $\gamma^{-1}$ when
$\gamma$ is small.
\begin{theorem}
\label{jiao-th2}
Suppose \eqref{1-2}-\eqref{1-5} hold. Let $u \in C^4 (\bar{\Omega})$ be an admissible solution
of \eqref{appro}. Then we have
\begin{equation}
\label{jiao-3-1-1}
\sup_{\bar{\Omega}} |D^2 u| \leq C \max\{\gamma, 1\} (1 + \sup_{\partial \Omega} |D^2 u|),
\end{equation}
for some constant $C$ depending on
$|u|_{C^1 (\bar{\Omega})}$, $|\ul u|_{C^2 (\bar{\Omega})}$
and other known data.
In particular, if $0 \leq \gamma \leq 1$, \eqref{jiao-3-1} holds
for the constant $C$ depending on
$|u|_{C^1 (\bar{\Omega})}$, $|\ul u|_{C^2 (\bar{\Omega})}$
and other known data.
\end{theorem}
\begin{proof}
In the proof of Theorem \ref{jw-th2}, let $\zeta \equiv 1$ and $\phi = \frac{\delta}{2} |D u|^2 + b (\ul u - u)$,
where $\delta$ and $b$ are positive constants to be chosen. Note that
\[
\phi_i = \delta u_i u_{ii} + b (\ul u - u)_i
\]
and
\[
\phi_{ii} = \delta u^2_{ii} + \delta u_j u_{jii} + b (\ul u - u)_{ii}.
\]
We have
\begin{equation}
\label{ib-2}
\phi_i^2 \leq C \delta^2 u_{ii}^2 + C b^2.
\end{equation}
Therefore, by \eqref{ib-1},
\begin{equation}
\label{ib-3-9}
\begin{aligned}
\mathcal{L} \phi
     \geq \delta F^{ii} u_{ii}^2 + \gamma \delta \sum u_{jj}^2 \sum F^{ii} - C \delta
        + b \mathcal{L} (\ul u - u).
\end{aligned}
\end{equation}
We can derive from \eqref{3-2}-\eqref{3-4} and \eqref{ib-2} that
\begin{equation}
\label{ib-3}
\begin{aligned}
\mathcal{L} \phi
     \leq \frac{C}{u_{11}} + C \delta^2 (F^{ii} u_{ii}^2 + \gamma \sum u_{jj}^2 \sum F^{ii})
        + C b^2 \sum F^{ii}.
\end{aligned}
\end{equation}
Combining \eqref{ib-3-9} and \eqref{ib-3}, we obtain
\begin{equation}
\label{ib-4}
(\delta - C \delta^2) (F^{ii} u_{ii}^2 + \gamma \sum u_{jj}^2 \sum F^{ii})
   + b \mathcal{L} (\ul u - u) \leq C \delta + \frac{C}{u_{11}} + C b^2 \sum F^{ii}.
\end{equation}
We may assume that $\delta$ is sufficiently small such that $(\delta - C \delta^2) > \delta/2$.
let ${\mu} = \lambda (D^2 \ul u (x_0) + \gamma \triangle \ul u (x_0))$,
$\lambda = \lambda (D^2 u (x_0) + \gamma \triangle u (x_0))$.
As in the gradient estimates, we consider two cases:
(i) $|\nu_\mu- \nu_\lambda| < \beta$ and (ii) $|\nu_\mu- \nu_\lambda| \geq \beta$,
where $\beta$ is as in \eqref{eq2-16}.

In case (i), we see that \eqref{eq2-22} holds. Thus, by \eqref{eq4-0-10}, \eqref{ib-4}, \eqref{eq2-22}
and \eqref{eq0-1}, we have
\begin{equation}
\label{ib-5}
(2 n \max \{\gamma, 1\})^{-1}\frac{\delta \beta}{2 \sqrt{n}} |\lambda|^2 \sum F^{ii}
    \leq C \delta + \frac{C}{u_{11}} + C b^2 \sum F^{ii}.
\end{equation}
We may assume $|\lambda| \geq R$ for $R$ sufficiently large such that \eqref{eq4-5} holds.
Therefore, by \eqref{eq4-5} and \eqref{ib-5}, we have
\begin{equation}
\label{ib-6}
(2 n \max \{\gamma, 1\})^{-1} \Big(\frac{\delta \beta}{4 \sqrt{n}} |\lambda|^2 \sum F^{ii}
   + \frac{\delta \beta b_0}{4 \sqrt{n}} |\lambda|\Big)
    \leq C \delta + \frac{C}{u_{11}} + C b^2 \sum F^{ii}.
\end{equation}
It follows that
\[
|\lambda| \leq \frac{8 n \sqrt{n} \max \{\gamma, 1\}}{\delta \beta}
    \max \Big\{C b, \frac{C \delta}{b_0} + \frac{C}{R b_0}\Big\}.
\]
Note that we do not determine $\delta$ and $b$ right now.

In case (ii), we can choose $b$ sufficiently small such that $\theta b > C b^2$,
where $\theta$ is as in \eqref{gj-C160'}. We can choose such $b$ and a smaller $\delta$
such that $\theta b > C \delta$. Applying Lemma \ref{ma-lemma-C10}, we can derive
from \eqref{ib-4} that
\[
u_{11} \leq \frac{C}{\theta b - C \delta}
\]
and \eqref{jiao-3-1-1} holds.
\end{proof}

\medskip

\section{Boundary estimates for second derivatives}

In this section we consider the estimates for the second order derivatives on
the boundary $\partial \Omega$. As usual, the construction of barrier functions
plays a key role. For any fixed $x_0 \in \Omega$, we may assume that $x_0$ is the
origin of $\mathbb{R}^n$ with the positive $x_n$ axis in the interior normal
direction to $\partial \Omega$ at the origin. Let $d (x)$
be the distances from $x \in \bar{\Omega}$ to $\partial \Omega$,
and set
\[
\Omega_\delta = \{x \in \Omega: |x| < \delta\}.
\]
Suppose near the origin, the boundary $\partial \Omega$
is represented by
\begin{equation}\label{4-6}
x_{n}=\rho(x')=\frac{1}{2}\sum_{\alpha,\beta<n}B_{\alpha\beta}x_{\alpha}x_{\beta}+O(|x'|^{3})
\end{equation}
for some $C^{\infty}$ smooth function $\rho$, where $x'=(x_{1},...,x_{n-1})$.
For $x\in\partial\Omega$ near the origin, let
\[
T_\alpha = T_\alpha(x)=\partial_{\alpha}
  +\sum_{\beta<n}B_{\alpha\beta}(x_{\beta}\partial_{n}-x_{n}\partial_{\beta}), \;\;
  \mbox{ for } \alpha < n
\]
and $T_n = \partial_n$. We have (see \cite{CNS3})
\[
\mathcal{L} T_\alpha u = T_\alpha (\psi + \varepsilon \eta (\psi)).
\]
It follows that
\begin{equation}
\label{bd-2}
|\mathcal{L} T_\alpha (u - \varphi)| \leq C \Big(1 + \sum F^{ii}\Big)
\end{equation}
and
\begin{equation}
\label{bd-3}
|T_\alpha (u - \varphi)| \leq C |x|^2 \mbox{ on } \partial \Omega \cap \bar{\Omega}_\delta
    \mbox{ for } \alpha < n
\end{equation}
when $\delta$ is sufficiently small since $u = \varphi$ on $\partial \Omega$.

To proceed we choose smooth unit orthonormal vector fields $e_1, \ldots, e_{n}$ in $\Omega_\delta$ such that
when restricted to $\partial \Omega$, $e_1, \ldots, e_{n - 1}$ are tangential
and $e_n$ is normal to $\partial \Omega$. Let $e_i (x) = (\xi^i_1 (x), \ldots, \xi^i_n (x))$,
$\nabla_i u = \xi^i_k D_k u$, $\nabla_{ij} u = \xi^i_l \xi^j_k D_{kl} u$
and $\nabla^2 u = \{\nabla_{ij} u\}$ in $\Omega_\delta$.
We may assume $\xi^i_j (0) = \delta_{ij}$. In particular, $\lambda (D^2 u) = \lambda (\nabla^2 u)$
and
\[
f (\lambda (\nabla^2 u + \gamma \triangle u I)) = f (\lambda (D^2 u + \gamma \triangle u I)).
\]
By straightforward calculations, we have
\begin{equation}
\label{bs-1}
|\mathcal{L} \nabla_k (u - \varphi)| \leq C \Big(1 + \sum F^{ii} + \sum f_i |\widehat{\lambda}_i|
  + \gamma |\widehat{\lambda}| \sum F^{ii}\Big),
\end{equation}
where $\widehat{\lambda} = \lambda (D^2 u) = \lambda (\nabla^2 u)$.
Let $\widehat{F}^{ij} = \xi^i_l \xi^j_k F^{kl}$. We see that $\{\widehat{F}^{ij}\}$ is positive
definite with eigenvalues $f_1, \ldots, f_n$ and when $\nabla^2 u$ is diagonal so is $\{\widehat{F}^{ij}\}$.

We shall use the following barrier function
\begin{equation}
\label{bd-0}
\Psi = \frac{1}{\delta^2} \Big(A_1 (u - \ul u) + t d - \frac{N}{2} d^2
    + A_3 |x|^2\Big) - A_2 (u - \ul u) - A_4 \sum_{l < n} |\nabla_l (u - \varphi)|^2,
\end{equation}
where $A_1$, $A_2$, $A_3$, $A_4$, $t$ and $N$ are positive constants satisfying $A_1 > 2 A_2$.
\begin{theorem}
\label{barrier}
Suppose that \eqref{1-2}-\eqref{1-5} hold. Let $h \in C
(\bar{\Omega}_{\delta})$ satisfy $h \leq \overline{C} |x|^2$ on $\bar{\Omega}_{\delta}
\cap \partial \Omega$ and $h \leq \overline{C}$ on $\bar{\Omega}_{\delta}$. Then for
any positive constant $K$ there exist uniform positive constants $t,
\delta$ sufficiently small, and $A_1$, $A_2$, $A_3$, $N$
sufficiently large such that $\Psi \geq h$ on $\partial
\Omega_{\delta}$ and
\begin{equation}
\label{eq3-5}
\mathcal{L} \Psi \leq - K
\Big(1 + \sum F^{ii}\Big)  \;\; \mbox{in $\Omega_{\delta}$}.
\end{equation}
\end{theorem}
\begin{proof}
Let $v = t d - \frac{N d^2}{2}$. Firstly, we note that
\begin{equation}
\label{bd-1}
\begin{aligned}
\mathcal{L} v = \,& (t - N d) F^{ij} (d_{ij} + \gamma \triangle d \delta_{ij})
   - N F^{ij} (d_i d_j + \gamma |D d|^2 \delta_{ij})\\
      \leq \,& C_0 (t + N d) \sum F^{ii}
        - N F^{ij} d_i d_j - \gamma N \sum F^{ii}
\end{aligned}
\end{equation}
since $|D d| \equiv 1$. Similar to Proposition 2.19 in \cite{Guan14a}, we have
\begin{equation}
\label{bs-2}
\widehat{F}^{ij} \nabla_{il} u \nabla_{jl} u \geq \frac{1}{2} \sum_{i \neq r} f_i \widehat{\lambda}_i^2
\end{equation}
and
\begin{equation}
\label{bs-6}
 \sum _{l < n} \sum_{k=1}^{n} (\nabla_{lk} u)^2 \geq \frac{1}{2} \sum_{i \neq r} \widehat{\lambda}_i^2
\end{equation}
for some index $r$.
It follows from \eqref{bs-1}, \eqref{bs-2} and \eqref{bs-6} that
\begin{equation}
\label{bd-5}
\begin{aligned}
\sum_{l < n} \mathcal{L} |\nabla_l (u - \varphi)|^2
    = \,& 2 \sum_{l < n} F^{ij} \Big((\nabla_l (u - \varphi))_i (\nabla_l (u - \varphi))_j \\
        & + \sum_{l < n} \gamma |D \nabla_l (u - \varphi)|^2 \delta_{ij}\Big)
        + 2 \nabla_l (u - \varphi) \mathcal{L} \nabla_l (u - \varphi) \\
    \geq \,&  
          - C \Big(1 + \sum F^{ii} + \sum f_i |\widehat{\lambda}_i|
  + \gamma |\widehat{\lambda}| \sum F^{ii}\Big)\\
      & + \sum_{l < n} \widehat{F}^{ij} \nabla_{il} u \nabla_{jl} u
          + \frac{\gamma}{2} \sum _{l < n} \sum_{k=1}^{n} (\nabla_{lk} u)^2 \sum F^{ii}\\
         \geq \,& \frac{1}{2}\sum_{i \neq r} f_i \widehat{\lambda}_i^2
            + \frac{\gamma}{4} \sum_{i \neq r} \widehat{\lambda}_i^2 \sum f_i\\
     & - C \Big(1 + \sum F^{ii} + \sum f_i |\widehat{\lambda}_i|
  + \gamma |\widehat{\lambda}| \sum F^{ii}\Big).
\end{aligned}
\end{equation}

For any $x \in \Omega_\delta$, let $\mu = \lambda (D^2 \ul u (x) + \gamma \triangle \ul u (x) I)$,
$\lambda = \lambda (D^2 u (x) + \gamma \triangle u (x) I)$ and $\beta$ be as in \eqref{eq2-16}.
We consider two cases: (i) $|\nu_\mu- \nu_\lambda| < \beta$ and (ii) $|\nu_\mu- \nu_\lambda| \geq \beta$.

In case (i), we see that \eqref{eq2-22} holds. It follows that
\begin{equation}
\label{bs-12}
\sum_{i \neq r} f_i \widehat{\lambda}_i^2 \geq \frac{\beta}{\sqrt{n}} \sum_{i \neq r} \widehat{\lambda}_i^2 \sum f_i
\end{equation}
and
\begin{equation}
\label{bs-4}
\mathcal{L} v \leq - \frac{N \beta}{2 \sqrt{n}} \sum f_i
\end{equation}
provided $t$ and $\delta$ are sufficiently small since $D d \equiv 1$.

We first assume $|\lambda| \geq R$
for $R$ sufficiently large.
If $\widehat{\lambda}_r \leq 0$, we have $\sum_{i \neq r} \widehat{\lambda}_i > - \widehat{\lambda}_r$ since
$\triangle u > 0$. It follows that
\[
\sum_{i \neq r} \widehat{\lambda}_i^2 \geq c_0 \widehat{\lambda}_r^2
\]
for some unform constant $c_0 > 0$. Therefore, by \eqref{eq4-0-10}, there exists unform
positive constants $c_1$ and $c_2$ such that
\begin{equation}
\label{bd-6}
\sum_{i \neq r} \widehat{\lambda}_i^2 \geq c_1 \sum \widehat{\lambda}_i^2 \geq c_2 |\lambda|^2.
\end{equation}
Combining \eqref{eq2-22}, \eqref{eq0-1}, \eqref{eq4-5}, \eqref{bd-5}, \eqref{bs-4} and \eqref{bd-6},
\begin{equation}
\label{bs-5}
\begin{aligned}
\mathcal{L} \Psi \leq \,& - \frac{N \beta}{2 \sqrt{n} \delta^2} \sum f_i + \frac{C A_3}{\delta^2} \sum f_i
    - \frac{\beta c_2}{2 \sqrt{n}} A_4 |\lambda|^2 \sum f_i\\
      & + C A_4 \Big(1 + \sum f_i + |\widehat{\lambda}| \sum f_i \Big)\\
      \leq \,& - \frac{N \beta}{2 \sqrt{n} \delta^2} \sum f_i - \frac{\beta c_2 b_0}{4 \sqrt{n}} A_4 |\lambda|
         + \frac{C A_3}{\delta^2} \sum f_i + C A_4 \Big(1 + \sum f_i\Big)\\
            \leq \,& \Big(- \frac{N \beta}{2 \sqrt{n} \delta^2} + \frac{C A_3}{\delta^2} + C A_4\Big)\sum f_i
               - A_4
\end{aligned}
\end{equation}
provided $|\lambda| > R$ and $R$ is sufficiently large.

If $\widehat{\lambda}_r > 0$, we have
\begin{equation}
\label{bd-8}
\begin{aligned}
\mathcal{L} u = \,& F^{ij} u_{ij} + \gamma \triangle u \sum F^{ii}
   = \sum f_i \widehat{\lambda}_i + \gamma \sum \widehat{\lambda}_i \sum F^{ii}\\
    \geq \,& - \sum_{i \neq r} f_i |\widehat{\lambda}_i| - \gamma \sum_{i \neq r} |\widehat{\lambda}_i| \sum F^{ii}
    + f_r \widehat{\lambda}_r + \gamma \widehat{\lambda}_r \sum F^{ii}.
\end{aligned}
\end{equation}
Note that for each $\sigma > 0$ and each $1 \leq i \leq n$,
\begin{equation}
\label{bs-8}
\widehat{\lambda}_i^2 \geq 2 \sigma |\widehat{\lambda}_i| - \sigma^2.
\end{equation}
It follows that
\begin{equation}
\label{bs-18}
\frac{A_4}{2}\sum_{i \neq r} f_i \widehat{\lambda}_i^2 \geq 2 A_2 \sum_{i \neq r} f_i |\widehat{\lambda}_i|
    - \frac{2 A_2^2}{A_4} \sum_{i \neq r} f_i
\end{equation}
by letting $\sigma = 2 A_2/A_4$ and that
\begin{equation}
\label{bs-22}
\frac{A_4}{4}\sum_{i \neq r} \widehat{\lambda}_i^2 \geq 2 A_2 \sum_{i \neq r} |\widehat{\lambda}_i|
   - \frac{4 A_2^2}{A_4}
\end{equation}
by letting $\sigma = 4 A_2/A_4$. Therefore, by \eqref{bd-5}, \eqref{bd-8},
\eqref{bs-18} and \eqref{bs-22}, we find
\begin{equation}
\label{bs-3}
\begin{aligned}
& \mathcal{L} \Big(A_2 (u - \ul u) + A_4 \sum_{l < n} |\nabla_l (u - \varphi)|^2\Big)\\
  \geq \,& A_4 \Big(\frac{1}{2}\sum_{i \neq r} f_i \widehat{\lambda}_i^2
      + \frac{\gamma}{4} \sum_{i \neq r} \widehat{\lambda}_i^2 \sum f_i\Big) - C A_2 \sum f_i\\
        & + A_2 (f_r \widehat{\lambda}_r + \gamma \widehat{\lambda}_r \sum f_i)
          - A_2 \Big(\sum_{i \neq r} f_i |\widehat{\lambda}_i| + \gamma \sum_{i \neq r} |\widehat{\lambda}_i| \sum f_i \Big)\\
    & - C A_4 \Big(1 + \sum F^{ii} + \sum f_i |\widehat{\lambda}_i|
  + \gamma |\widehat{\lambda}| \sum f_i\Big) \\
  \geq \,& (A_2 - C A_4) \Big(\sum f_i |\widehat{\lambda}_i| + \gamma |\widehat{\lambda}| \sum f_i\Big)
     - C A_4 \Big(1 + \sum f_i\Big)\\
       & - \Big(C A_2 + \frac{2 A_2^2}{A_4} + \frac{4 A_2^2}{A_4} \gamma\Big) \sum f_i.
\end{aligned}
\end{equation}
Therefore, by \eqref{eq2-22} and \eqref{eq4-5}, we have
\begin{equation}
\label{bs-7}
\begin{aligned}
\mathcal{L} \Psi \leq \,& - \frac{N \beta}{2 \sqrt{n} \delta^2} \sum f_i
  - (A_2 - C A_4) \Big(\sum f_i |\widehat{\lambda}_i| + \gamma |\widehat{\lambda}| \sum f_i\Big)\\
   & + C \Big(A_2 + \frac{A_2^2}{A_4} + \frac{A_3}{\delta^2}\Big) \sum f_i
         + C A_4 \Big(1 + \sum f_i\Big)\\
      \leq \,& - \frac{N \beta}{2 \sqrt{n} \delta^2} \sum f_i
          - \frac{\beta (A_2 - C A_4)}{\sqrt{n}} |\lambda| \sum f_i \\
         & + C \Big(A_2 + \frac{A_2^2}{A_4} + \frac{A_3}{\delta^2}\Big) \sum f_i
             + C A_4 \Big(1 + \sum f_i\Big)\\
\leq \,& - \frac{N \beta}{2 \sqrt{n} \delta^2} \sum f_i
          - \frac{\beta (A_2 - C A_4) b_0}{\sqrt{n}} \\
         & + C \Big(A_2 + \frac{A_2^2}{A_4} + \frac{A_3}{\delta^2}\Big) \sum f_i
             + C A_4 \Big(1 + \sum f_i\Big).
\end{aligned}
\end{equation}

Now we assume $|\lambda| \leq R$. We see that \eqref{gd-0} holds. Thus, by \eqref{bs-4}
and \eqref{bd-5}, we have
\begin{equation}
\label{bs-9}
\begin{aligned}
\mathcal{L} \Psi \leq \,& - \frac{N \beta}{2 \sqrt{n}\delta^2} \sum f_i + C \Big(A_2 + \frac{A_3}{\delta^2} + A_4\Big) \Big(1 + \sum f_i\Big)\\
 \leq \,& - \frac{N \beta}{4 \sqrt{n}\delta^2} \sum f_i - \frac{N \beta \delta_0}{4 \sqrt{n}\delta^2}
   + C \Big(A_2 + \frac{A_3}{\delta^2} + A_4\Big) \Big(1 + \sum f_i\Big).
\end{aligned}
\end{equation}

Now we fix $A_3 > A_4 > K$ such that
\begin{equation}
\label{bs-15}
A_3 > A_4 \sup_{\bar{\Omega}} \sum_{l < n} |\nabla_l (u - \varphi)|^2 + \overline{C}.
\end{equation}

In case (ii), we see from Lemma \ref{ma-lemma-C10} that \eqref{eq0-0} holds.
We deal with two cases as before: $\widehat{\lambda}_r > 0$
and $\widehat{\lambda}_r \leq 0$.

If $\widehat{\lambda}_r > 0$, similar to \eqref{bs-7}, we obtain
\begin{equation}
\label{bs-14}
\begin{aligned}
\mathcal{L} \Psi
  \leq \,& - \theta \frac{A_1}{\delta^2} \Big(1 + \sum f_i\Big)
    + C_0 (t + N d) \sum f_i\\
      & + \frac{C A_3}{\delta^2} \sum f_i + C A_4 \Big(1 + \sum f_i\Big)
        + C \Big(A_2 + \frac{A_2^2}{A_4}\Big) \sum f_i\\
         & - (A_2 - C A_4) \Big(\sum f_i |\widehat{\lambda}_i| + \gamma |\widehat{\lambda}| \sum f_i\Big).
\end{aligned}
\end{equation}

If $\widehat{\lambda}_r \leq 0$, similar to \eqref{bd-8},
\begin{equation}
\label{bs-10}
- \mathcal{L} u
    \geq - \sum_{i \neq r} f_i |\widehat{\lambda}_i| - \gamma \sum_{i \neq r} |\widehat{\lambda}_i| \sum F^{ii}
    - f_r \widehat{\lambda}_r - \gamma \widehat{\lambda}_r \sum F^{ii}.
\end{equation}
Similar to \eqref{bs-3}, we have, for any $B > 0$,
\begin{equation}
\label{bs-11}
\begin{aligned}
& \mathcal{L} \Big(- B (u - \ul u) + A_4 \sum_{l < n} |\nabla_l (u - \varphi)|^2\Big)\\
  \geq \,& (B - C A_4) \Big(\sum f_i |\widehat{\lambda}_i| + \gamma |\widehat{\lambda}| \sum f_i\Big)
     - C A_4 \Big(1 + \sum f_i\Big)\\
       & - \Big(C B + \frac{2 B^2}{A_4} + \frac{4 B^2}{A_4} \gamma\Big) \sum f_i.
\end{aligned}
\end{equation}
Thus, choosing $B = A_2$, we can see from \eqref{eq0-0}, \eqref{bd-1} and \eqref{bs-11} that
\begin{equation}
\label{bs-21}
\begin{aligned}
\mathcal{L} \Psi
  \leq \,& - \theta \Big(\frac{A_1}{\delta^2} - 2 A_2\Big)\Big(1 + \sum f_i\Big)
    + C_0 (t + N d) \sum f_i\\
      & + \frac{C A_3}{\delta^2} \sum f_i + C A_4 \Big(1 + \sum f_i\Big)
        + C \Big(A_2 + \frac{A_2^2}{A_4}\Big) \sum f_i\\
         & - (A_2 - C A_4) \Big(\sum f_i |\widehat{\lambda}_i| + \gamma |\widehat{\lambda}| \sum f_i\Big).
\end{aligned}
\end{equation}

Now we choose $A_2 \gg A_4$ such that
\begin{equation}
\label{bs-16}
\frac{\beta (A_2 - C A_4) b_0}{\sqrt{n}} - C A_4 > K
\end{equation}
in \eqref{bs-7} and $A_2 - C A_4 > 0$ in \eqref{bs-14} and \eqref{bs-21}.
Next, we fix $N$ sufficiently large such that
\begin{equation}
\label{bs-17}
\frac{N \beta}{2 \sqrt{n} \delta^2} - \frac{C A_3}{\delta^2} - C A_4 > K
\end{equation}
in \eqref{bs-5},
\begin{equation}
\label{bs-19}
\frac{N \beta}{2 \sqrt{n} \delta^2} - C \Big(A_2 + \frac{A_2^2}{A_4} + \frac{A_3}{\delta^2}\Big) - C A_4 > K
\end{equation}
in \eqref{bs-7} and
\begin{equation}
\label{bs-20}
\frac{N \beta}{4 \sqrt{n}\delta^2} \min\{1, \delta_0\} - C \Big(A_2 + \frac{A_3}{\delta^2} + A_4\Big) > K
\end{equation}
in \eqref{bs-9}.

We may assume that $t$ and $\delta$ is sufficiently small such that $\theta A_1/\delta^2 > 2 C_0 (t + Nd)$.
Therefore, in case (ii), by \eqref{bs-14} and \eqref{bs-21}, we have
\begin{equation}
\label{bs-13}
\begin{aligned}
\mathcal{L} \Psi
  \leq \,& \Big(- \frac{A_1 \theta}{2 \delta^2} + \frac{C A_3}{\delta^2}
      + C A_2 + C \frac{A_2^2}{A_4} + C A_4 \Big) \Big(1 + \sum f_i\Big).
\end{aligned}
\end{equation}
Finally, we may choose $A_1$ large enough to obtain \eqref{eq3-5}.
Furthermore, $v \geq 0$ in $\Omega_\delta$ when $\delta \leq 2t/N$. Therefore,
we can ensure $\Psi \geq h$ on $\partial
\Omega_{\delta}$.
\end{proof}

Now we are ready to establish the boundary estimates for second order derivatives.
Firstly, it is easy to obtain a bound independent of $\varepsilon$ for pure
tangential second order derivatives on the boundary
\begin{equation}\label{jw-1}
|u_{\xi \eta}|_{C^0 (\partial \Omega)} \leq C
\end{equation}
from the boundary condition in \eqref{1-1}, where $\xi$ and $\eta$
are unit tangential vector fields on $\partial\Omega$.

For the estimates of mixed second order derivatives, we see
from \eqref{bd-2}, \eqref{bd-3} and \eqref{eq3-5} that
\[
\mathcal{L} (\Psi \pm T_\alpha (u - \varphi)) \leq 0 \mbox{  in } \Omega_\delta
\]
and
\[
\Psi \pm T_\alpha (u - \varphi) \geq 0 \mbox{  on } \partial \Omega_\delta
\]
for $\alpha < n$. It follows that
\begin{equation}\label{jw-2}
|u_{\xi \nu}|_{C^0 (\partial \Omega)} \leq C,
\end{equation}
where $\xi$ is any unit tangential vector on $\partial \Omega$ and
$\nu$ is the unit inner normal of $\partial\Omega$.
It suffices to establish an upper bound for the double normal derivative
on the boundary $\partial \Omega$ since $(1 + n \gamma) \triangle u \geq 0$.

As in \cite{ITW2004}, let $T=\{T_{i}^{j}\}$ be a skew-symmetric matrix,
such that $e^{T}$ is orthogonal, where $T_{i}^j$ is the entry of $i^{th}$
row and $j^{th}$ column of $T$.
Let $\tau = (\tau_{1}, \ldots, \tau_{n})$ be a vector field in $\Omega$ given by
\[
\tau_{i} = T_{i}^{j} x_{j} + a_i,  \;\; i = 1, \ldots ,n,
\]
where $a_i$ is a constant.
Denote $u_{\tau\tau} = \tau_{i}\tau_{j}u_{ij}$ and $u_{(\tau)(\tau)}=(u_{\tau})_{\tau}=\tau_{i}\tau_{j}u_{ij}+(\tau_{i})_{j}\tau_{j}u_{i}$.
Similar to Lemma 2.1 of \cite{ITW2004} we can prove the following lemma.
\begin{lemma}
\label{jw-lem1}
We have
\[
\mathcal{L} (u_{(\tau)(\tau)}) \geq (F[U])_{(\tau)(\tau)}.
\]
\end{lemma}
\begin{proof}
Similar to Lemma 2.1 of \cite{ITW2004}, by the skew-symmetry of
$T$, we have
\begin{equation}\label{4-3}
F^{ij}(T_{i}^{k} u_{kj\tau} + T_{j}^{k} u_{ki\tau})
   = -F^{ij,st} (T_{i}^{k} u_{kj} + T_{j}^{k}u_{ki}) U_{st\tau}
\end{equation}
and
\begin{equation}\label{4-4}
\begin{aligned}
& F^{ij}(2T_{i}^{k}T_{j}^{l}u_{kl}+T_{i}^{k}T_{k}^{l}u_{lj}+T_{j}^{k}T_{k}^{l}u_{li})\\
 = & -F^{ij,st}(T_{i}^{k}u_{kj}+T_{j}^{k}u_{ki})(T_{s}^{k}u_{kt}+T_{t}^{k}u_{ks}).
\end{aligned}
\end{equation}
Note that
\[
(u_{(\tau) (\tau)})_{ij} = u_{ij (\tau)(\tau)} - 2 T^k_i u_{kj\tau}
  - 2 T^k_j u_{ki\tau} + 2 T^s_i T^t_j u_{st} + T^t_j T^s_t u_{si} + T^t_i T^s_t u_{sj}.
\]
We find
\begin{equation}\label{4-5}
\begin{aligned}
\mathcal{L} \,& (u_{(\tau)(\tau)})
 = F^{ij} u_{ij(\tau)(\tau)} + \gamma \sum F^{ii}(\Delta u)_{(\tau)(\tau)}\\
     & + F^{ij} \Big( 2 T_{i}^{s} T_{j}^{t} u_{st} + T_{j}^{t} T_{t}^{s} u_{si} +
        T_{i}^{t} T_{t}^{s} u_{sj} - 2 T_{i}^{k} u_{kj\tau} - 2 T_{j}^{k} u_{ki\tau}\Big)\\
     & + \gamma \sum F^{ii} \Big(2 T_{l}^{s} T_{l}^{t} u_{st}
        + 2 T_{l}^{t} T_{t}^{s} u_{sl} - 4 T_{l}^{k} u_{kl\tau}\Big).
\end{aligned}
\end{equation}
Next, since $T$ is skew-symmetric,
\[
2 T_{l}^{s} T_{l}^{t} u_{st} + 2 T_{l}^{t}T_{t}^{s}u_{sl} - 4 T_{l}^{k}u_{kl\tau} = 0.
\]
It follows from \eqref{4-3}, \eqref{4-4} and \eqref{4-5} that
\[
\begin{aligned}
\mathcal{L} (u_{(\tau)(\tau)})
 = \,& F^{ij} u_{ij(\tau)(\tau)} + \gamma \sum F^{ii}(\Delta u)_{(\tau)(\tau)}\\
   & - F^{ij,st} (T_{i}^{k} u_{kj} + T_{j}^{k} u_{ki}) (T_{s}^{k}u_{kt} + T_{t}^{k} u_{ks})\\
   & + F^{ij,st} \Big((T_{i}^{k} u_{kj} + T_{j}^{k} u_{ki}) U_{st\tau} + (T_{s}^{k}
       u_{kt} + T_{t}^{k }u_{ks}) U_{ij\tau}\Big).
\end{aligned}
\]
Note that
\[
(u_{\tau})_{ij}=u_{ij\tau}-T_{i}^{k}u_{kj}-T_{j}^{k}u_{ki}.
\]
We obtain
\[
\begin{aligned}
\mathcal{L} (u_{(\tau)(\tau)})
   = \,& (F[U])_{(\tau)(\tau)} - F^{ij,st}(\gamma\delta_{ij}(\Delta u)_{\tau}+(u_{\tau})_{ij})
     (\gamma\delta_{st}(\Delta u)_{\tau}+(u_{\tau})_{st})\\
   \geq \,& (F[U])_{(\tau)(\tau)}.
\end{aligned}
\]
\end{proof}

Now we establish the double normal derivative estimates. We
may assume the origin is a boundary point such that $e_{n} = (0, \ldots, 0, 1)$
is the unit inner normal there
and denote
$$M=\sup_{x\in\partial\Omega}D_{\nu\nu}u(x).$$
where $\nu$ is the unit inner normal of $\partial\Omega$ at $x\in\partial\Omega$.
Without loss of generality, we assume
\[
M=\sup_{\partial\Omega}|D^{2}u|,
\]
and
\begin{equation}\label{4-7}
\sup_{\bar{\Omega}} |D^2 u| \leq C M,
\end{equation}
for some uniform constant $C \geq 1$.
By Lemma \ref{jw-lem1}, we have
\begin{equation}\label{4-8}
\mathcal{L}(T_\alpha^2 u)\geq  T_\alpha^2 (F[U]) = T_\alpha^2 (\psi+ \varepsilon \eta (\psi)) \geq -C.
\end{equation}
According to \cite{ITW2004}, we see
\begin{equation}\label{4-9}
w(x) \equiv T_\alpha^2 u (x) - T_\alpha^2 u (0)\leq C_{0}(|x'|^{2} + M |x'|^{4}) \equiv h(x')
\end{equation}
for $x\in\partial\Omega$ with $|x'|\leq r_{0}$.

Let
\[
\overline{\Psi} = \frac{1}{\delta^4} \Big(A_1 (u - \ul u) + t d - \frac{N}{2} d^2
    + A_3 |x|^4\Big) - A_2 (u - \ul u) - \sum_{l < n} |\nabla_l (u - \varphi)|^2.
\]
Using the same arguments to Theorem \ref{barrier}, by \eqref{4-8}, \eqref{4-7} and \eqref{4-9},
we can show that
there exists positive constants $A_1$, $A_2$, $A_3$, $N$ sufficiently large and $t$, $\delta$
sufficiently small such that
\[
\mathcal{L}\Big(w - h(x') - M \overline{\Psi} \Big)\geq 0 \;\;
    \mbox{in} ~ \Omega_{\delta}
\]
and
\[
w - h(x') - M \overline{\Psi} \leq 0 \;\; \mathrm{on} ~ \partial\Omega_{\delta}.
\]
It follows from the maximum principle that
\begin{equation}\label{4-10}
w - h(x') - M \overline{\Psi} \leq 0 \;\; \mathrm{on} ~ \bar{\Omega}_{\delta}.
\end{equation}
It follows that
\begin{equation}\label{4-11}
w \leq C M (u - \underline{u} + d) + C M |x|^4 + C  \;\; \mathrm{on} ~ \bar{\Omega}_{\delta}.
\end{equation}
Therefore, for each small $\sigma > 0$, we can find a positive constant
$\delta_1^4 = C \sigma < \delta^4$ such that
\begin{equation}\label{4-11-0}
w \leq C M (u - \underline{u} + d) +\frac{\sigma}{2} M + C  \;\; \mathrm{on} ~ \bar{\Omega}_{\delta_1}
\end{equation}
and
\begin{equation}\label{4-12}
\mathcal{L} h \leq (\sqrt{\sigma} M + C) \sum F^{ii}  \;\; \mathrm{on} ~ \bar{\Omega}_{\delta_1}.
\end{equation}
Next, there exists a positive constant $\delta_2 < \delta_1$
such that
\[
C (u - \underline u + d) \leq \frac{\sigma}{2} \;\;\;\; \mathrm{on} ~ \Omega - \widehat{\Omega}_{\delta_2},
\]
where $\widehat{\Omega}_{\delta_2} \equiv \{x \in \Omega: \mathrm{dist} (x, \partial \Omega) > \delta_2\}$,
since $|D (\underline u - u)| \leq C$ independent of $\varepsilon$.
Hence we can derive from \eqref{4-11-0} that
\begin{equation}\label{4-13}
w \leq \sigma M + C  \;\;\;  \mathrm{on} \; \bar{\Omega}_{\delta_1} \cap (\Omega - \widehat{\Omega}_{\delta_2}).
\end{equation}
On the other hand, by \eqref{3-1}, there exists a positive constant $C$ depending on
$\gamma^{-1}$, $\delta_2$ and $|u|_{C^1 (\bar{\Omega})}$ such that
\begin{equation}\label{4-14}
|w| \leq C \;\;\;\; \mathrm{ in } ~ \widehat{\Omega}_{\delta_2}.
\end{equation}
Thus, there exists a positive constant $C_\sigma$ depending on $\sigma$ and other known data such that
\[
|w| \leq \sigma M + C_\sigma \;\;\;  \mathrm{on} \; \bar{\Omega}_{\delta_1}.
\]
Similar to Theorem \ref{barrier}, we can find positive constants
$A_1$, $A_2$, $A_3$, $t$ and $N$ such that
\[
\mathcal{L} (w - (\sigma M + C_\sigma) \Psi') \geq 0 \mbox{  in } \Omega_{\delta_1}
\]
and
\[
w - (\sigma M + C_\sigma) \Psi' \leq 0 \mbox{  on } \partial \Omega_{\delta_1} \cap \Omega,
\]
where
\[
\Psi' = \frac{1}{\delta_1^2} \Big(A_1 (u - \ul u) + t d - \frac{N}{2} d^2
    + A_3 |x|^2\Big) - A_2 (u - \ul u) - \sum_{l < n} |T_l (u - \varphi)|^2.
\]
Note that that main terms to control $\sum F^{ii}$ in Theorem \ref{barrier} are
$u - \ul u$ and $\frac{N}{2} d^2$. We may assume that $\sigma$ is sufficiently small.
Reviewing the proof of Theorem \ref{barrier},
we can find constants $t'$ sufficiently small and $N'$ sufficiently large
such that
\begin{equation}
\label{bd-9}
\mathcal{L} (u - \ul u + t'd - \frac{N'}{2} d^2) \leq - \varepsilon_1 \sum F^{ii}
\mbox{  in } \Omega_{\delta_1}
\end{equation}
for some positive constant $\varepsilon_1$ and
\begin{equation}
\label{bd-10}
u - \ul u + t'd - \frac{N'}{2} d^2 \geq 0
\mbox{  on } \bar{\Omega}_{\delta_1}.
\end{equation}
By \eqref{4-12}, \eqref{bd-9} and \eqref{bd-10},
we can choose a constant $A$ sufficiently large such that
\[
\mathcal{L} (w - (\sigma M + C_\sigma) \Psi' - A (\sqrt{\sigma} M + C) w - h) \geq 0 \mbox{  in } \Omega_{\delta_1},
\]
and
\[
w - (\sigma M + C_\sigma) \Psi' - A (\sqrt{\sigma} M + C) w - h \leq 0  \mbox{  on } \partial \Omega_{\delta_1},
\]
where $w = u - \ul u + t' d - \frac{N'}{2} d^2$.
Thus, by the maximum principle again, we have
\[
w \leq (C \sqrt{\sigma} M + C_\sigma)
  \big(u - \underline{u} + d + |x|^{2}\big) + h (x') \;\;  \mbox{ on } \bar{\Omega}_{\delta_1}.
\]
Therefore we obtain
\begin{equation}\label{4-15}
(T_\alpha^2 u)_{n}(0)\leq C\sqrt{\sigma} M+C_\sigma \mbox{  for each } \alpha < n.
\end{equation}
It follows that
\[
u_{n(\xi)(\xi)}\leq C\sqrt{\sigma} M+C_\sigma \;\; \mathrm{on} ~ \partial\Omega
\]
for any tangential unit vector field $\xi$ on $\partial\Omega$.

Now choose a new coordinate system and suppose the maximum $M$ is attained at the origin $0\in\partial\Omega$,
and near the origin $\partial\Omega$ is given by (\ref{4-6}). By the Taylor expansion, we have
\[
u_{n}(x)\leq u_{n}(0)+\sum_{\alpha<n}u_{n\alpha}(0)x_{\alpha}+ (C\sqrt{\sigma} M+C_\sigma) |x'|^{2}
\]
for $x \in \partial \Omega$ near the origin, where $u_{n\alpha}(0)$ is bounded by (\ref{jw-2}). Denote
\[
g \equiv u_{n}(x)-u_{n}(0)-\sum_{\alpha<n}u_{n\alpha}(0)x_{\alpha}- (C\sqrt{\sigma} M+C_\sigma) |x'|^{2}.
\]
In \eqref{bd-0}, we may choose another group of positive constants $A_1$, $A_2$, $A_3$, $t$,
$N$ and $\delta$ such that
\[
\mathcal{L}\Big(g - (\sqrt{\sigma} M+C_\sigma) \Psi \Big)\geq 0 \;\; \mathrm{in} ~ \Omega_{\delta}
\]
and
\[
g - (\sqrt{\sigma} M+C_\sigma) \Psi \leq 0 \;\; \mathrm{on} ~ \partial\Omega_{\delta}.
\]
Applying the maximum principle again we obtain
\[
M=u_{nn}(0)\leq C \sqrt{\sigma} M+C_\sigma.
\]
Choosing $\sqrt{\sigma}<1/2C$, we get a bound $M\leq C$ and \eqref{1-7} is proved.
\begin{remark}
We remark that in this paper, the condition that $\gamma > 0$ is only used to establish the interior
estimate \eqref{3-1}.
\end{remark}

\end{document}